\documentclass[12pt]{amsart}
\advance\hoffset by -0.5 truecm
\usepackage{amssymb}
\newtheorem{Theorem}{Theorem}[section]
\newtheorem{Lemma}[Theorem]{Lemma}
\newtheorem{Corollary}[Theorem]{Corollary}

\newtheorem{Proposition}[Theorem]{Proposition}

\newtheorem{Definition}[Theorem]{Definition}
\newtheorem{Remark}[Theorem]{Remark}

\def\V{\mbox{Var}}

\def\Z{{\mathbb Z}}
\def\R\re
\def\V{\bf V}

\def \re{{\mathbb R}}

\def \T{{\mathcal T}}
\def \C{{\mathbb C}}

\def \V{{\bf V}}

\def \A{{\mathbb A}}

\def \d{\mathbf d}

\begin{document}
\title[Transparent connections]{Transparent connections over negatively curved surfaces}

\author[G.P. Paternain]{Gabriel P. Paternain}
 \address{ Department of Pure Mathematics and Mathematical Statistics,
University of Cambridge,
Cambridge CB3 0WB, UK}
 \email {g.p.paternain@dpmms.cam.ac.uk}




\begin{abstract} Let $(M,g)$ be a closed oriented negatively curved surface. A unitary connection
on a Hermitian vector bundle over $M$ is said to be transparent if its parallel transport along the closed
geodesics of $g$ is the identity. We study the space of such connections modulo gauge and we prove
a classification result in terms of the solutions of certain PDE that arises
naturally in the problem. We also show a local uniqueness result for
the trivial connection and that there
is a transparent $SU(2)$-connection associated to each meromorphic function
on $M$.

\end{abstract}

\maketitle

\section{Introduction}
Let $(M,g)$ be a closed Riemannian manifold and let $E\to M$ be a Hermitian vector bundle of rank $n$ over $M$.
A unitary connection $\nabla$ on $E$ is said to be {\it transparent} if its parallel transport along every closed
geodesic of $g$ is the identity. These connections are ``ghosts" or ``invisible" from the point of view
of the closed geodesics of $g$.
Clearly, if $\nabla$ is transparent any other connection gauge equivalent to
it will also be transparent. The goal of the present paper is the study of transparent connections modulo gauge
transformations when $(M,g)$ is a closed oriented negatively curved surface.

The motivation for studying this problem comes from several a priori unrelated quarters. Transparent connections
on $S^2$ (and $\re \mathbb P^2$) arise in a natural way in the theory of integrable systems and solitons
when studying the Bogomolny equations $D\Phi=\star F$ in $(2+1)$-dimensional Minkowski space \cite{Wa1,Wa2}.
Here $\Phi$ is the Higgs field, $F$ is the curvature of the connection, $\star$ is the Hodge star operator
of the metric and $D$ is the induced connection on the endomorphism bundle.
The condition of having trivial holonomy along the closed geodesics of a compactified
space-like plane picks up finite dimensional families of solutions
and enables the use of methods from twistor theory over a compact twistor space.
\cite{A}. In fact, using a more refined twistor correspondence, L. Mason
has recently classified all transparent connections on $S^2$ and $\re \mathbb P^2$ with the standard round metric \cite{Ma1}.
For the case of $S^2$, his results say that the space of transparent connections modulo gauge
is in 1-1 correspondence with holomorphic vector bundles $W\to (\mathbb C\mathbb P^2)^*$ and positive definite
Hermitian metrics on $W$ restricted to the real slice $(\mathbb R\mathbb P^2)^*$. Similar results are obtained
for anti-self-dual Yang-Mills connections over $S^2\times S^2$ with split signature, see \cite{Ma}.

The problem of determining a connection from its parallel transport along geodesics is a natural integral-geometry problem
that can be considered also in the case of manifolds with boundary or $\re^d$ with appropriate decay
conditions at infinity. It arises for example when one considers the wave equation associated to the
Schr\"odinger equation with external Yang-Mills potential $A$ and the inverse problem of determining
the potential $A$ from the Dirichlet-to-Neumann map $\Lambda_{A}$. 
There are various results known for the integral-geometry problem. Local uniqueness theorems under various
assumptions on the connection or its curvature are proved by V.A. Sharafutdinov \cite{Sha}, R. Novikov \cite{No}
and D. Finch and G. Uhlmann \cite{FU}. A global uniqueness result for connections with compact support
is proved by G. Eskin in \cite{E}. In the case of $\re^2$, Novikov shows (building on the work of Ward previously 
mentioned) that global uniqueness may fail and in fact, his construction gives non-trivial transparent
connections over $\re \mathbb P^2$. He also shows global uniqueness (with
reconstruction) for $d\geq 3$ without assuming compact support.

As we mentioned before, in the present paper we will discuss transparent connections when the metric
is negatively curved, or more generally, when its geodesic flow is Anosov.
While our main focus here will be in the non-abelian case, we should mention that the abelian
case $n=1$ is also of interest, but it can be reduced to known results (see Theorem \ref{thm:abelian}) to
obtain that transparent connections, when they exist, are unique up to gauge equivalence.
The abelian case arises also when discussing positivity of entropy production 
in dissipative geodesic flows or thermostats \cite{DP}, thus showing that the problem of understanding transparent
connections also pops up naturally in dynamical systems and non-equilibrium statistical mechanics.

Our first result (Theorem \ref{thm:livtop}) asserts that not all bundles over a surface of genus \texttt{g}
carry transparent connections. In fact we show that a complex vector bundle $E$ over $M$ admits a
transparent connection if and only if $2-2\texttt{g}$ divides its first Chern class $c_{1}(E)$.
This result, and subsequent ones, are based on the classical Livsic theorem for non-abelian cocycles
which is recalled in Section \ref{sec:livsic}.

One of the obvious differences with the abelian case is the appearance of the following ghosts.
Let $K$ be the canonical line bundle and $K^s$ with $s\in \Z$ be its tensor powers (if $s=0$ we get the trivial bundle). The powers $K^s$ for $s\neq 0$ carry the Levi-Civita connection
which we denote by $\nabla_{\ell}^{s}$. If $s=0$ we understand that
this is the trivial connection. Note that the Levi-Civita connection
on $TM$ ($=K^{-1}$) is transparent, since the parallel transport 
along
a closed geodesic $\gamma$ must fix $\dot{\gamma}(0)$ and consequently any vector orthogonal
to it since the parallel transport is an isometry and the surface is
 orientable. Thus any $\nabla_{\ell}^{s}$ is transparent.
Given an $n$-tuple of integers $S:=(s_{1},\dots,s_{n})$, the connection
\[\nabla_{\ell}^{S}:=\nabla_{\ell}^{s_{1}}\oplus\cdots \oplus\nabla_{\ell}^{s_{n}}\]
defines a transparent unitary connection on the bundle
$E_{S}:=K^{s_{1}}\oplus\cdots\oplus K^{s_{n}}$.
Clearly $c_{1}(E_{S})=(2\texttt{g}-2)(s_{1}+\dots+s_{n})$ and any complex vector bundle $E$
supporting a transparent connection is isomorphic to $E_S$ for $S$ such that
$c_{1}(E)=c_{1}(E_{S})$.

Now let $E$ be a Hermitian vector bundle and consider a unitary isomorphism
$\tau:E\to E_{S}$, where $S$ is such that $c_{1}(E)=c_{1}(E_{S})$.
The unitary connection $\tau^*\nabla_{\ell}^S$ is a transparent
connection on $E$ and its gauge equivalence class, denoted by $[S]$, is independent
of $\tau$. Note that $[S_1]=[S_2]$ if and only if
$S_1$ and $S_2$ coincide up to a permutation. 
However, as we shall see below, these will not be the only ghosts.

Given two transparent connections $\nabla^1$ and $\nabla^2$ write $\nabla^2=\nabla^1+A$
where $A\in \Omega^{1}(M,\mbox{\rm ad}\,E)$. Let $\pi: SM\to M$ be the unit circle bundle
and $X$ the vector field of the geodesic flow of the metric. The Livsic theorem
will provide solutions $u\in \Omega^{0}(SM,\mbox{\rm Aut}\,\pi^*E)$ of
$D_{X}u+Au=0$, where $D$ is the connection on the bundle of endomorphisms of $\pi^*E$ induced
by $\nabla^1$. Inspired by the methods in \cite{GK} we will show that these solutions
must have a {\it finite} Fourier expansion (cf. Theorem \ref{thm:deg}), in other words, $u$
must be a polynomial in the velocities. The degree of this polynomial
will allow us to define a distance function on the set of transparent connections modulo
gauge. As a consequence of these results we will derive the following local uniqueness statement
for the trivial connection on the trivial bundle:

\medskip

\noindent{\bf Theorem A}. {\it Consider the trivial bundle and let $\nabla$ be a transparent
connection with curvature $F_{\nabla}$. Let $K<0$ be the Gaussian curvature of the surface
and suppose that the Hermitian matrix $\pm i\star F_{\nabla}(x)-K(x)\,\mbox{\rm Id}$
is positive definite for all $x\in M$. Then $\nabla$ is gauge equivalent
to the trivial connection.}

\medskip

Thus, if a transparent connection has small enough curvature, it must be gauge equivalent
to the trivial connection. Note that Theorem A is sharp, since a ghost $\nabla\in [S]$
with $\sum_{k=1}^{n} s_{k}=0$, $s_{k}\in \{0,\pm 1\}$ and $S\neq 0$ has
$\pm i\star F_{\nabla}(x)-K(x)\,\mbox{\rm Id}$ positive semi-definite for all $x\in M$.

Our second main result is a classification of the set $\T$ of transparent connections on $E$ modulo gauge in terms of the solutions of certain
non-linear PDE that arises naturally in the problem.
In order to describe this PDE, recall that the unit sphere bundle $SM$ of $M$
has a canonical frame $\{X,H,V\}$ where $X$ is the geodesic vector field, $V$ is the vertical vector
field and $H=[V,X]$ is the horizontal vector field. 
Let $f:SM\to  \mathfrak{u}(n)$ be a smooth
function, where $\mathfrak{u}(n)$ denotes the set of all $n\times n$ skew-Hermitian matrices.
Consider the PDE:
\begin{equation}
H(f)+VX(f)=[X(f),f].
\label{eq:keypde}
\end{equation}
Observe that the set $\mathcal H$ of solutions to (\ref{eq:keypde}) is invariant under the action of $U(n)$ given by $f\mapsto q^{-1}fq$, where $q\in U(n)$.

We shall say that two functions $f,h:SM\to \mathfrak u(n)$ are
{\it $V$-cohomologous} if there exists a smooth function $u:SM\to U(n)$
such that $f=u^{-1}V(u)+u^{-1}h u$.

Given a constant matrix $c\in \mathfrak u(n)$ with $e^{2\pi c}=\mbox{\rm Id}$
we consider the $U(n)$-invariant
subset $\mathcal H_{c}\subset \mathcal H$ given by those solutions $f$
which are $V$-cohomologous to $c$. The set $\mathcal H_{c}$ only
depends on $\mbox{\rm tr}(c)$ (see Lemma \ref{lemma:traces}).

\medskip

\noindent{\bf Theorem B}. {\it Let $E$ be a Hermitian vector bundle over a closed oriented surface of genus} $\texttt{g}$ {\it whose geodesic flow is Anosov. Suppose
that} $2-2\texttt{g}$ {\it divides $c_{1}(E)$ and let $c\in \mathfrak u(n)$ be a constant matrix with $e^{2\pi c}=\mbox{\rm Id}$ and} $c_{1}(E)=(2\texttt{g}-2)\,\mbox{\rm tr}(ic)$. {\it Then, there is a 1-1 correspondence between $\T$ and
$\mathcal H_{c}/U(n)$.}

\medskip

In the abelian case $n=1$, it is not hard to see that the only solutions to
(\ref{eq:keypde}) are the constants provided $K<0$.
 This can be shown using the energy estimates method (the Pestov identity)
which also gives some information about (\ref{eq:keypde}) for $n\geq 2$. This is discussed
at the end of the paper, Subsection \ref{sub:last}.

For $n\geq 2$, the constant solutions in $\mathcal H_{c}$ correspond precisely
to the Levi-Civita ghosts $[S]$, but as we mentioned before, there are other
ghosts besides $[S]$ and these have to come from non-constants solutions
to (\ref{eq:keypde}). To see that this is the case we consider
functions $f$ which only depend on the base point $x$. Under such assumption,
it is easy to see that (\ref{eq:keypde}) turns into $2\star df=[df,f]$,
which only depends on the conformal class of the metric $g$.
We discuss this equation in Subsection \ref{sub:sol} for the case of
$SU(2)$ and we show that its non-zero solutions correspond precisely
with the set of holomorphic maps $f:M\to \mathbb C\mathbb P^1$. We
also show that all these maps are $V$-cohomologous to the zero matrix.
In this way, via Theorem B, we virtually obtain as many $SU(2)$-transparent
connections on the trivial bundle (modulo gauge) as meromorphic functions on $M$; these are
all the transparent connections at distance one from the trivial connection
(cf. Corollary \ref{cor:last} for the precise statement).

There are several questions which are worthy of further consideration.
In particular, it would be interesting to exhibit elements in $\mathcal H_c$ which have
dependence on the velocities. It seems that one can deal with this issue using an appropriate
B\"acklund transformation, but we leave it as the subject of a future paper.
The inclusion of a Higgs field
$\Phi$ and the problem of understanding transparent pairs
$(\nabla,\Phi)$ is also of interest, but it will also be discussed elsewhere.

\medskip

{\it Acknowledgements:} I am very grateful to N. Dairbekov, M. Dunajski, L. Mason, 
R. Novikov, V. Sharafutdinov, R. Spatzier and G. Uhlmann for several useful
comments and discussions related to this paper.
I am also greateful to the referee for several comments and corrections
that improved the presentation.

\section{The Livsic theorem for non-abelian cocycles}
\label{sec:livsic}
Let $X$ be a closed manifold and $\phi_t:X\to X$ a smooth
transitive Anosov flow.  Recall that $\phi_t$
is Anosov if there
is a continuous splitting
$TX=E^0\oplus E^{u}\oplus E^{s}$, where $E^0$ is the flow direction, and 
there are constants $C>0$ and $0<\rho<1<\eta$ such that 
for all $t>0$ we have
\[\|d\phi_{-t}|_{E^{u}}\|\leq C\,\eta^{-t}\;\;\;\;\mbox{\rm
and}\;\;\;\|d\phi_{t}|_{E^{s}}\|\leq C\,\rho^{t}.\]
It is very well known that the geodesic flow of a closed
negatively curved Riemannian manifold is a transitive Anosov
flow.

Let $G$ be a compact Lie group; for the purposes
of this paper it is enough to think of $G$ as $U(n)$.

\begin{Definition} A $G$-valued cocycle over the flow $\phi_t$ is a map
$C:X\times\re\to G$ that satisfies
\[C(x,t+s)=C(\phi_{t}x,s)\,C(x,t)\]
for all $x\in X$ and $s,t\in\re$.
\label{def1}
\end{Definition}

In this paper the cocycles will always be smooth. In this case $C$ is
determined by its infinitesimal generator $A:X\to\mathfrak g$ given by
\[A(x):=\left.\frac{d}{dt}\right|_{t=0}C(x,t).\]
The cocycle can be recovered from $A$ as the unique solution to
\[\frac{d}{dt}C(x,t)=dR_{C(x,t)}(A(\phi_{t}x)),\;\;\;C(x,0)=\mbox{\rm Id},\]
where $R_{g}$ is right translation by $g\in G$.

\begin{Theorem}[The smooth Livsic periodic data theorem]
Let $C$ be a smooth cocycle such that $C(x,T)=\mbox{\rm Id}$ whenever
$\phi_{T}x=x$. Then, there exists a smooth function $u:X\to G$ such that
\[C(x,t)=u(\phi_{t}x)u(x)^{-1}.\]
\label{livsic}
\end{Theorem}

The existence of a H\"older continuous function $u$ (assuming $A$
is H\"older) was proved by Livsic \cite{L1,L2}.
Smoothness of $u$ was proved by Ni\c tic\u a and T\"or\"ok \cite{NT}.

In our applications we will need to consider non trivial vector bundles. 
Suppose $E$ is a rank $n$ Hermitian vector bundle over $X$ and $\phi_t:X\to X$ is as above,
a smooth transitive Anosov flow.

\begin{Definition} A cocycle over $\phi_t$ is an action of $\re$
by bundle automorphisms which covers $\phi_t$. In other words, for each
$(x,t)\in X\times \re$, we have a unitary map $C(x,t):E_{x}\to E_{\phi_{t}x}$
such that $C(x,t+s)=C(\phi_{t}x,s)\,C(x,t)$.
\end{Definition}

If $E$ admits a unitary trivialization $f:E\to X\times \C^n$, then
\[f\,C(x,t)\,f^{-1}(x,a)=(\phi_{t}x,D(x,t)a),\] where $D:X\times\re\to U(n)$
is a cocycle as in Definition \ref{def1}.

Let $E^*$ denote the dual vector bundle to $E$. If $E$ carries
a Hermitian metric $h$, we have a conjugate isomorphism $\ell_h: E\to E^*$, which induces a Hermitian metric $h^*$ on $E^*$. Given a cocycle
$C$ on $E$, $C^*:=\ell_{h}\,C\,\ell_{h}^{-1}$ is a cocycle on
$(E^*,h^*)$.

\begin{Proposition} Let $E$ be a Hermitian vector bundle over $X$ such
that $E\oplus E^*$ is a trivial vector bundle. Let $C$ be a smooth
cocycle on $E$ such that $C(x,T)=\mbox{\rm Id}$ whenever
$\phi_T x=x$. Then $E$ is a trivial vector bundle.
\label{trivial}
\end{Proposition}

\begin{proof} As explained above, the cocycle $C$ on $E$ induces
a cocycle $C^*$ on $E^*$.
On the trivial vector bundle $E\oplus E^*$ we consider the cocycle
$C\oplus C^*$. Clearly $C\oplus C^*(x,T)=\mbox{\rm Id}$ everytime
that $\phi_T x=x$. Choose a unitary trivialization
$f:E\oplus E^*\to X\times \C^{2n}$ and write
\[f\,C\oplus C^*(x,t)\,f^{-1}(x,a)=(\phi_{t}x,D(x,t)a).\]
By Theorem \ref{livsic}, there exists a smooth function
$u:X\to U(2n)$ such that $D(x,t)=u(\phi_{t}x)u^{-1}(x)$.
Since $\phi_t$ is a transitive flow, we may choose $x_0\in X$
with a dense orbit and without loss of generality we may suppose
that $u(x_0)=\mbox{\rm Id}$. Let 
$$\{e_{1}(x_0),\dots,e_{n}(x_{0})\}$$
be a unitary frame at $E_{x_{0}}$.
Write $f(x_0,e_{i}(x_{0}))=(x_{0},a_{i})$, where $a_{i}\in \C^{2n}$ and let
\[e_{i}(x):=f^{-1}(x,u(x)a_{i}).\]
Clearly at every $x\in X$, $\{e_{1}(x),\dots,e_{n}(x)\}$ is a smooth unitary
$n$-frame of $E_{x}\oplus E_{x}^{*}$.
We claim that in fact $e_{i}(x)\in E_x$ for all $x\in X$. This, of course,
implies the triviality of $E$. Note that
\[e_{i}(\phi_{t}x_{0})=f^{-1}(\phi_{t}x_{0},u(\phi_{t}x_{0})a_{i})=
f^{-1}(\phi_{t}x_{0},D(x_{0},t)a_{i})=C\oplus C^{*}(x_{0},t)e_{i}(x_{0}).\]
But $e_{i}(x_{0})\in E_{x_{0}}$, thus $e_{i}(\phi_{t}x_{0})\in E_{\phi_{t}x_{0}}$. It follows that $e_{i}(x)\in E_{x}$ for a dense set of points in $X$.
By continuity of $e_{i}$, $e_{i}(x)\in E_{x}$ for all $x\in X$.

\end{proof}

\begin{Remark}{\rm The hypothesis of $E\oplus E^*$ being trivial is not needed in Proposition \ref{trivial}.
Ralf Spatzier has informed the author that it is possible to adapt the proof of the
usual Livsic periodic data theorem to show directly that $E$ is trivial. However, this weaker version
is all that we will need in this paper.

A proof of the measurable Livsic theorem for bundles (which we do not use here) may be found in
\cite{GS}.

}
\end{Remark}

\section{Transparent connections and the Livsic theorem}
\label{tliv}
Let $M^d$ be a closed orientable Riemannian manifold whose geodesic flow $\phi_t$ 
is Anosov. The geodesic flow acts on the unit sphere bundle $SM$ and we let
$\pi:SM\to M$ be the footpoint projection.

Let $E\to M$ be a Hermitian vector bundle of rank $n$ and let $\nabla$
be a unitary connection on $E$. Given a geodesic $\gamma:\re\to M$
with initial conditions $(x,v)\in SM$, we can consider the parallel
transport of $\nabla$ along $\gamma$. The parallel
transport $P_{x,\gamma(t)}:E_{x}\to E_{\gamma(t)}$ is an isometry and defines
a smooth cocycle $C$ over the geodesic flow on the Hermitian vector bundle
over $SM$ given by the pull-back bundle $\pi^* E$.
The connection $\nabla$ is transparent if and only if $C(x,v,T)=\mbox{\rm Id}$
every time that $\phi_{T}(x,v)=(x,v)$.

\subsection{Arbitrary bundles over an Anosov surface} Suppose $d=2$.
In this case, complex vector bundles $E$ over $M$ are classified topologically
by the first Chern class $c_{1}(E)\in H^{2}(M,\Z)=\Z$. Since
$c_{1}(E^*)=-c_{1}(E)$ and $c_1$ is additive with respect to direct sums,
we see that $E\oplus E^*$ is the trivial bundle and therefore we will be able
to apply Proposition \ref{trivial}. In fact we will show:

\begin{Theorem} Let $M$ be a closed orientable surface of genus 
$\mbox{\tt g}$
whose geodesic flow is Anosov. A complex vector bundle $E$ over $M$ admits a
transparent connection if and only if $2-2\tt{g}$ divides $c_{1}(E)$.
\label{thm:livtop}
\end{Theorem}

\begin{proof} Suppose $E$ admits a transparent connection. As explained above
we may apply Proposition \ref{trivial} to deduce that $\pi^*E$ is
a trivial bundle and since $c_{1}(\pi^*E)=\pi^* c_{1}(E)$ we conclude that
$\pi^*c_{1}(E)=0$. Consider now the Gysin sequence of the unit circle
bundle $\pi:SM\to M$,
\[0\to H^{1}(M,\Z)\stackrel{\pi^*}{\to}H^{1}(SM,\Z)\stackrel{0}{\to}
H^{0}(M,\Z)\stackrel{\times(2-2\texttt{g})}{\longrightarrow} H^{2}(M,\Z)\stackrel{\pi^*}{\to}
H^{2}(SM,\Z)\to\cdots.                 \]
We see that $\pi^*c_{1}(E)=0$ if and only if $c_1(E)$ is in the image
of the map $H^0(M,\Z)\to H^{2}(M,\Z)$ given by cup product with the
Euler class of the unit circle bundle. Equivalently, $2-2\texttt{g}$ must divide
$c_1(E)$.

Let $K$ be the canonical line bundle of $M$. We can think of $K$ as the cotangent bundle to $M$; it has $c_{1}(K)=2\texttt{g}-2$. The tensor powers $K^s$ of $K$ (positive and negative) generate all possible
 line bundles with first Chern class
divisible by $2-2\texttt{g}$ and they all carry the unitary connection induced
by the Levi-Civita connection of the Riemannian metric on $M$. 
All these connections are clearly transparent. Topologically, all complex
vector bundles over $M$ whose whose first Chern class
is divisible by $2-2\texttt{g}$ are of the form $K^s\oplus \varepsilon$,
where $\varepsilon$ is the trivial vector bundle. Since the trivial
connection on the trivial bundle is obviously transparent,
it follows that every complex vector bundle whose first Chern class
is divisible by $2-2\texttt{g}$ admits a transparent connection.

\end{proof}

\subsection{Arbitrary bundles over an Anosov 3-manifold} Suppose $M$ is a closed 3-manifold
whose geodesic flow is Anosov. Complex vector bundles $E$ over $M$ are also classified
topologically by $c_{1}(E)\in H^{2}(M,\Z)$, hence as in the two dimensional case,
$E\oplus E^{*}$ is the trivial bundle. Thus if $E$ admits a transparent connection,
Proposition \ref{trivial} implies that $\pi^{*}E$ is trivial. However now the argument
with the Gysin sequence that we explained in the proof of Theorem \ref{thm:livtop} shows that
$\pi^{*}:H^{2}(M,\Z)\to H^{2}(SM,\Z)$ is injective and thus $c_{1}(E)=0$.
Therefore if $E$ admits a transparent connection, it must be trivial.
This shows that the problem for 3-manifolds is in some sense simpler than the problem
for surfaces, at least, there are no obvious transparent connections besides the trivial one.

\subsection{The abelian case} The goal of this subsection is to show the following result.

\begin{Theorem} Let $M$ be a closed orientable Riemannian manifold whose geodesic flow
is Anosov and let $E$ be a Hermitian line bundle over $M$. 
Then, any two transparent connections are gauge equivalent.
\label{thm:abelian}
\end{Theorem}

\begin{proof} Let $\nabla^1$ and $\nabla^2$ be transparent connections. We may write
$\nabla^2=\nabla^1+A$, where $A\in \Omega^{1}(M,\mbox{\rm ad}\,E)$.
Since $E$ is a line bundle, $A=i\theta$, where $\theta$ is a real-valued
1-form in $M$. Since $\nabla^1$ and $\nabla^2$ are transparent,
\begin{equation}
\int_{\gamma}\theta\in 2\pi\,\Z
\label{trans}
\end{equation}
for every closed geodesic $\gamma$.
Consider the cocycle over $\phi_t$, $C:SM\times\re\to S^1$ defined as follows.
Given $(x,v)\in SM$, let $\gamma:\re\to M$ be the unique geodesic with initial
conditions $(x,v)$. Set
\[C(x,v,t):=\exp\left(i\int_{0}^t\theta_{\gamma(s)}(\dot{\gamma}(s))\,ds \right).\]
By (\ref{trans}), the cocycle $C$ has the property that $C(x,v,T)=1$, every time
that $\phi_T(x,v)=(x,v)$, thus by Theorem \ref{livsic}, there is a smooth
function $u:SM\to S^1$ such that $C(x,v,t)=u(\phi_{t}(x,v))u^{-1}(x,v)$.
If we differentiate this equality with respect to $t$ and set $t=0$ we obtain
\begin{equation}
ui\theta=du(X),
\label{ec}
\end{equation}
where $X$ is the geodesic vector field.
The function $u$ gives rise to a real-valued closed 1-form in $SM$ given by
$\varphi:=\frac{du}{iu}$. Since $\pi^*:H^1(M,\re)\to H^{1}(SM,\re)$
is an isomorphism (this follows easily from the Gysin sequence, since
$M$ cannot be the 2-torus), there exists a closed 1-form $\omega$ in
$M$ and a smooth function $f:SM\to\re$ such that
\[\varphi=\pi^*\omega+df.\]
When this equality is applied to $X$ and combined with (\ref{ec}) one obtains
\[\theta_{x}(v)-\omega_{x}(v)=df(X(x,v))\]
for all $(x,v)\in SM$. This cohomological equation is actually equivalent
-via the classic Livsic theorem for $\re$-values cocycles- to saying that
\[\int_{\gamma}\theta-\omega=0\]
for every closed geodesic $\gamma$. It is known that this implies
that $\theta-\omega$ is exact. This was proved by V. Guillemin and D. Kazhdan
\cite{GK} for surfaces of negative curvature, by C. Croke and Sharafutdinov
\cite{CS} for arbitrary manifolds of negative curvature and by N.S. Dairbekov and
Sharafutdinov \cite{DS} for manifolds whose geodesic flows is Anosov.

If $\theta-\omega$ is exact, $\theta$ must be closed and by (\ref{trans}),
$[\theta]/2\pi\in H^{1}(M,\Z)$. Thus there exists a smooth function
$g:M\to S^1$ such that $\theta=dg/ig$. This is precisely the statement
that $\nabla^1$ and $\nabla^2$ are gauge equivalent.

\end{proof}

\section{Setting up the Fourier analysis}

Let $E$ be a complex Hermitian vector bundle over $M$ and let $\nabla$
be a unitary connection on $E$. If $\pi:SM\to M$ denotes the canonical 
projection, then $\nabla$ induces a unitary connection
on the pull-back bundle $\pi^*E$ that we denote by $\pi^*\nabla$.
This pull-back connection induces in turn a unitary connection
on the bundle $\mbox{\rm End}\,\pi^*E$ of complex endomorphisms of $\pi^*E$,
which we denote by $D$. Note that $\mbox{\rm End}\,\pi^*E$ naturally
inherits a Hermitian metric determined by the trace $\mbox{\rm tr}(u\,w^*)$
where $u,w\in \Omega^{0}(SM,\mbox{\rm End}\,\pi^*E)$. This Hermitian metric
together with the Liouville measure $\mu$ of $SM$ combine to give
an $L^2$-inner product of sections
\[\langle u,w^*\rangle=\int_{SM} \mbox{\rm tr}(u\,w^*)\,d\mu.\]

Let $F_{\nabla}\in \Omega^2(M,\mbox{\rm ad}\,E)$ be the curvature
of $\nabla$. Then $F_{\pi^*\nabla}\in \Omega^2(SM,\mbox{\rm ad}\,\pi^*E)$
is given by $F_{\pi^*\nabla}=\pi^*F_{\nabla}$ and
 $F_{D}\in \Omega^2(SM,\mbox{\rm ad}\,\mbox{\rm End}\,\pi^*E)$
is given by $F_{D}=[F_{\pi*\nabla},\,\cdot]$.
Note that if $\star$ denotes the Hodge star operator of the metric, then
$\star F_{\nabla}\in \Omega^{0}(M,\mbox{\rm ad}\,E)$.

The vertical vector field $V$ and the connection $D$ induce a first order differential operator
$$D_{V}:\Omega^{0}(SM,\mbox{\rm End}\,\pi^*E)\to
\Omega^{0}(SM,\mbox{\rm End}\,\pi^*E)$$
which in fact is independent of $\nabla$, for if $\nabla'$ is another connection and we write $\nabla'=\nabla+A$, then $D'=D+[\pi^{*}A,\cdot]$ and
$D'_{V}=D_{V}$ since $\pi^*A(V)=0$.

Note that $-iD_{V}$ is self-adjoint, since $V$ preserves the Lioville measure $\mu$.
Indeed, observe that
the compatibility of $D$ with the Hermitian metric implies
$V\langle u,w\rangle=\langle D_{V}u,w\rangle+\langle u,D_{V}w\rangle$.
Since the integral of $V\langle u,w\rangle$ with respect to $\mu$ vanishes,
$(D_{V})^*=-D_{V}$.
We also have an orthogonal decomposition
$$L^{2}(SM,\mbox{\rm End}\,\pi^*E)=\bigoplus _{n\in\Z}H_{n}$$
such that $-iD_{V}$ acts as $n\,\mbox{\rm Id}$ on $H_n$. To see this, triangulate $M$ in such a way
that both $SM\to M$ and $E\to M$ are trivial over each face $M_r$ of the triangulation.
Since $L^{2}(SM,\mbox{\rm End}\,\pi^*E)$ is isomorphic to $\oplus_{r}L^{2}(SM_{r},\mbox{\rm End}\,\pi^*E)$
we are reduced to the case of both bundles being trivial in which case the claim is clear because
$D_{V}u=V(u)$, where $u$ is a matrix valued function on $M_{r}\times S^1$.

 Following Guillemin and
Kazhdan in \cite{GK} we introduce the following first order differential operators 
$$\eta_{+},\eta_{-}:\Omega^{0}(SM,\mbox{\rm End}\,\pi^*E)\to
\Omega^{0}(SM,\mbox{\rm End}\,\pi^*E)$$ given by

\[\eta_{+}:=\frac{D_{X}-i\,D_{H}}{2}\]
\[\eta_{-}:=\frac{D_{X}+i\,D_{H}}{2},\]
where $H=[V,X]$. We recall the other two structure equations of the Riemannian metric:
$X=-[V,H]$ and $[X,H]=KV$, where $K$ is the Gaussian curvature of the surface.

The next lemma describes the commutation relations between these
operators.

\begin{Lemma} We have
\begin{align*}
[-iD_{V},\eta_{+}]&=\eta_{+},\\
[-iD_{V},\eta_{-}]&=-\eta_{-},\\
[\eta_{+},\eta_{-}]&=\frac{i}{2}\left(K\,D_{V}+[\star F_{\nabla},\,\cdot]\right).
\end{align*}
\label{commeta}
\end{Lemma}

\begin{proof} In order to prove the lemma we need the following
general preliminary observation: if $U$ and $W$ are vector fields
in $SM$ then
\begin{equation}
F_{D}(U,W)=[D_{U},D_{W}]-D_{[U,W]}.
\label{curva}
\end{equation}
As noted before 
\begin{equation}
F_{D}=[\pi^*F_{\nabla},\,\cdot].
\label{culift}
\end{equation}
Thus for any vector field $U$, $F_{D}(V,U)=0$ ($V$ is vertical) and hence
$[D_{V},D_{U}]=D_{[V,U]}$. If we now take $U=X,H$ and we use
the commutation relations $[V,X]=H$ and $[V,H]=-X$ we obtain
$[D_{V},D_{X}]=D_{H}$ and $[D_{V},D_{H}]=-D_{X}$. The first
two commutation relations in the lemma follow easily from this.
To prove the third relation note that
$2[\eta_{+},\eta_{-}]=i[D_{X},D_{H}]$. Using
(\ref{curva}) and (\ref{culift}) together with $[X,H]=K\,V$
we see that $[D_{X},D_{H}]=F_{D}(X,H)+K\,D_{V}=
[\star F_{\nabla},\cdot]+K\,D_{V}$ and the third commutation relation follows.

\end{proof}

Let us set $\Omega_{n}:=H_{n}\cap \Omega^{0}(SM,\mbox{\rm End}\,\pi^*E) $.
The first two commutation relations in the lemma imply
that $\eta_{+}:\Omega_{n}\to \Omega_{n+1}$ and
$\eta_{-}:\Omega_{n}\to \Omega_{n-1}$. It also follows
easily from the fact that $X$ and $H$ preserve
$\mu$ and the definitions 
that $\eta_{+}^{*}=-\eta_{-}$ and $\eta_{-}^{*}=-\eta_{+}$.
Indeed, like $D_{V}$, $D_{X}$ and $D_{H}$ are skew-Hermitian since both
$X$ and $H$ preserve the Liouville measure $\mu$.

\subsection{Modified operators} We will now modify the operators
$\eta_{+}$ and $\eta_{-}$ to suit our purposes.
Consider a second unitary connection $\nabla^{0}$ and
write $\nabla^0=\nabla+A$, where $A\in \Omega^{1}(M,\mbox{\rm ad}\,E)$.
We may regard $A$ and $\star A$ as elements of $\Omega^{0}(SM,\mbox{\rm ad}\,\pi^*E)$ and if we 
do so, then $D_{V}A=\star A$ since $(D_{V}A)(x,v)=A(x,iv)=(\star A)(x,v)$. This certainly implies that $D_{V}^{2}A=-A$.
Having this in mind, we decompose $A$ as $A=A_{-1}+A_{1}$ where
\[A_{1}:=\frac{A-iD_{V}A}{2}\in \Omega_{1},\;\;\;A_{-1}:=\frac{A+iD_{V}A}{2}\in \Omega_{-1}.\]
Observe that this decomposition corresponds precisely with the usual decomposition of 1-forms
on a surface:
\[\Omega^{1}(M,\mbox{\rm ad}\,E)\otimes \C=\Omega^{1,0}(M,\mbox{\rm ad}\,E)\oplus \Omega^{0,1}(M,\mbox{\rm ad}\,E),\]
given by the eigenvalues $\pm i$ of the Hodge star operator.

We now set $\mu_{+}:=\eta_{+}+A_{1}$ and $\mu_{-}:=\eta_{-}+A_{-1}$.
It is straightforward to check that $\mu_{+}:\Omega_{n}\to \Omega_{n+1}$ and
$\mu_{-}:\Omega_{n}\to \Omega_{n-1}$. It is also easy to check that
$\mu_{+}^{*}=-\mu_{-}$ and $\mu_{-}^{*}=-\mu_{+}$.

We will need the following auxiliary lemma.

\begin{Lemma} The following relation holds
\[\frac{i}{2}\star (\nabla A+A\wedge A)=\eta_{+}(A_{-1})-\eta_{-}(A_{1})
+A_{1}A_{-1}-A_{-1}A_{1}.\]

\label{auxiliar}
\end{Lemma}

\begin{proof}Using the definitions we derive
\[A_{1}A_{-1}-A_{-1}A_{1}=\frac{i}{2}(AD_{V}(A)-D_{V}(A)A),\]
\[\eta_{+}(A_{-1})-\eta_{-}(A_{1})=\frac{i}{2}(D_{X}D_{V}A-D_{H}A).\]
But it is easy to check that
\[\star(A\wedge A)=AD_{V}(A)-D_{V}(A)A,\]
and since
\begin{align*}
\star(\nabla A)&=D_{X}(\pi^*A)(X,H)\\
&=D_{X}(\pi^*A(H))-D_{H}(\pi^*A(X))-\pi^*A([X,H])\\
&=D_{X}D_{V}A-D_{H}A,
\end{align*}
the lemma is proved.

\end{proof}

The next lemma describes the commutator $[\mu_{+},\mu_{-}]$.

\begin{Lemma} Given $u\in  \Omega^{0}(SM,\mbox{\rm End}\,\pi^*E)$ we have
\[[\mu_{+},\mu_{-}]u=\frac{i}{2}\left(K\,D_{V}u+(\star F_{\nabla^{0}})\,u-u\,
(\star F_{\nabla})\right).\]

\end{Lemma}

\begin{proof} Using the definitions we compute
\[[\mu_{+},\mu_{-}]u=[\eta_{+},\eta_{-}]u+(\eta_{+}(A_{-1})-\eta_{-}(A_{1}))u
+(A_{1}A_{-1}-A_{-1}A_{1})u.\]
Using Lemmas \ref{commeta} and \ref{auxiliar} we obtain
\[[\mu_{+},\mu_{-}]u=\frac{i}{2}\left(K\,D_{V}u+[\star F_{\nabla},u]\right)
+\frac{i}{2}\star(\nabla A+A\wedge A)u.\]
The lemma is now a consequence of the fact that
$F_{\nabla^{0}}=F_{\nabla}+\nabla A +A\wedge A$.

\end{proof}

Let $|\cdot|$ stand for the $L^2$-norm in
$\Omega^{0}(SM,\mbox{\rm End}\,\pi^*E)$.

\begin{Corollary}  Given $u\in  \Omega^{0}(SM,\mbox{\rm End}\,\pi^*E)$ we have
\[|\mu_{+}u|^2=|\mu_{-}u|^2+\frac{i}{2}\left(\langle K\,D_{V}u,u\rangle
+\langle (\star F_{\nabla^{0}})u,u\rangle-\langle u(\star F_{\nabla}),u\rangle\right).\]
\label{cor:1}
\end{Corollary}

\begin{proof} Using that $\mu_{+}^{*}=-\mu_{-}$ and $\mu_{-}^{*}=-\mu_{+}$
we derive
\begin{align*}
|\mu_{+}u|^2&=\langle \mu_{+}u,\mu_{+}u\rangle\\
&=\langle (\mu_{+})^*\mu_{+}u,u\rangle\\
&=\langle -\mu_{-}\mu_{+}u,u\rangle\\
&=\langle -\mu_{+}\mu_{-}u,u\rangle+
\langle [\mu_{+},\mu_{-}]u,u\rangle\\
&=\langle (\mu_{-})^*\mu_{-}u,u\rangle+\langle [\mu_{+},\mu_{-}]u,u\rangle\\
&=|\mu_{-}u|^2+\langle [\mu_{+},\mu_{-}]u,u\rangle
\end{align*}
and the corollary follows from the previous lemma.

\end{proof}

\section{A distance between transparent connections}
Let $\nabla^1$ and $\nabla^2$ be two unitary connections. We may write
$\nabla^2=\nabla^1+A$, where $A\in \Omega^{1}(M,\mbox{\rm ad}\,E)$.

If $\pi:SM\to M$ denotes the canonical projection, we obtain unitary 
connections on the pull-back bundle $\pi^{*}E$ which are related by
\[\pi^*\nabla^{2}=\pi^*\nabla^{1}+\pi^* A,\]
where $\pi^*A\in \Omega^{1}(SM,\mbox{\rm ad}\,\pi^*E)$.
These connections on $\pi^*E$ induce in turn connections $D^1$ and $D^2$
on the bundle $\mbox{\rm End}\,\pi^*E$ of complex endomorphisms of $\pi^*E$
and are related by
\[D^{2}=D^{1}+[\pi^* A,\,\cdot].\]

Suppose now that both $\nabla^1$ and $\nabla^2$ are transparent.
As explained in Section \ref{tliv}, they induce smooth cocycles
$C_1$ and $C_2$ on $\pi^*E$. By Proposition \ref{trivial}, $\pi^*E$
is trivial and via a unitary trivilization
we may use the Livsic theorem \ref{livsic} to conclude that there exists
a smooth $u\in \Omega^{0}(SM,\mbox{\rm Aut}\,\pi^*E)$ such that
\[C_{2}(x,v,t)=u(\phi_t(x,v))C_{1}(x,v,t)u^{-1}(x,v).\]
Take $\xi\in E_{x}$. Since $C_{1}(x,v,t)\xi$ (resp. $C_{2}(x,v,t)\xi$)
is $\nabla^1$-parallel (resp. $\nabla^2$-parallel)
along the geodesic determined by $(x,v)$, if we apply
$\nabla^1$ to the previous equality and set $t=0$ we obtain at $(x,v)$:
\[-A\xi=(D_{X}u)u^{-1}\xi\]
where $D:=D^1$, and since this holds for all $\xi$ we derive
\begin{equation}
D_{X}u+Au=0.
\label{keyec}
\end{equation}

The main result that we will prove about the solutions $u$ of (\ref{keyec})
is that they have a {\it finite} Fourier expansion.

Given an element $u\in \Omega^{0}(SM, \mbox{\rm End}\,\pi^*E)$, we write
$u=\sum_{m\in\Z}u_{m}$, where $u_m\in \Omega_m$.
We will say that $u$ has degree $N$, if $N$ is the smallest
non-negative integer such that $u_{m}=0$ for all $m$ with $|m|\geq N+1$.

\begin{Theorem} Let $u\in \Omega^{0}(SM, \mbox{\rm End}\,\pi^*E)$
be a smooth solution to (\ref{keyec}). Then $u$ has degree $N<\infty$.
Moreover $N\leq l-1$ where $l$ is the smallest positive integer
such that the Hermitian operators
\[\mbox{\rm End}\,E_{x}\ni \alpha\mapsto -l\,K(x)\alpha
+(i\star F_{\nabla^{2}}(x))\alpha-\alpha(i\star F_{\nabla^1}(x)),\]
\[\mbox{\rm End}\, E_{x}\ni \alpha\mapsto -l\,K(x)\alpha
-(i\star F_{\nabla^{2}}(x))\alpha+\alpha(i\star F_{\nabla^1}(x))\]
are positive definite for all $x\in M$.
\label{thm:deg}
\end{Theorem}

\begin{proof} Since $D_{X}=\eta_{+}+\eta_{-}$, equation (\ref{keyec})  may be rewritten as
\[\mu_{+}(u)+\mu_{-}(u)=0.\]
Projecting onto $\Omega_m$-components we obtain
\begin{equation}
\mu_{+}(u_{m-1})+\mu_{-}(u_{m+1})=0
\label{eq:mu}
\end{equation}
for all $m\in \Z$.
Since $K<0$, there exists a positive integer $l$ such that the Hermitian operators
\[u\mapsto -l\,K u
+(i\star F_{\nabla^{2}})u-u(i\star F_{\nabla^1}),\]
\[u\mapsto -l\,K u
-(i\star F_{\nabla^{2}})u+u(i\star F_{\nabla^1})\]
are positive definite for all $x\in M$.
Using Corollary \ref{cor:1}, we can find a constant $c>0$ such that
\begin{equation}
|\mu_{+}(u_{m})|^2\geq |\mu_{-}(u_{m})|^2+c|u_{m}|^2
\label{ineq:1}
\end{equation}
for all $m\geq l$. There is also a constant $d>0$
such that
\begin{equation}
|\mu_{-}(u_{m})|^2\geq |\mu_{+}(u_{m})|^2+d|u_{m}|^2
\label{ineq:2}
\end{equation}
for all $m\leq -l$.
Combining (\ref{eq:mu}) and (\ref{ineq:1}) we obtain
\begin{equation}
|\mu_{+}(u_{m+1})|\geq |\mu_{+}(u_{m-1})|
\label{ineq:3}
\end{equation}
for all $m\geq l-1$. Similarly, it follows from (\ref{eq:mu}) and (\ref{ineq:2}) that
\begin{equation}
|\mu_{-}(u_{m-1})|\geq |\mu_{-}(u_{m+1})|
\label{ineq:4}
\end{equation}
for all $m\geq -l+1$.
Since the function $u$ is smooth, $\mu_{+}(u_{m})$ must tend to zero in the $L^2$-topology
as $m\to\infty$. It follow from (\ref{ineq:3}) that $\mu_{+}(u_{m})=0$ for $m\geq l-2$.
However, (\ref{ineq:1}) implies that $\mu_{+}$ is injective for $m\geq l$ and thus
$u_{m}=0$ for $m\geq l$. Similarly, using (\ref{ineq:2}) and (\ref{ineq:4}) we deduce that
$u_{m}=0$ for $m\leq -l$. This shows that $u$ has finite degree $N\leq l-1$.

\end{proof}

Let $\T$ denote the space of transparent connections modulo gauge
transformations. Using the previous theorem we can introduce a distance
function on $\T$ as follows.
Given $[\nabla^1],[\nabla^{2}]\in\T$ we set
$\d([\nabla^1],[\nabla^2]):=N$, where $N$ is the smallest degree
of $u\in \Omega^{0}(SM,\mbox{\rm Aut}\,\pi^*E)$ which solves
$D_{X}u+Au=0$. It is easy to check that this definition
does not depend on the chosen representatives as we now explain. Let
$\varphi, \psi \in \Omega^{0}(M,\mbox{\rm Aut}\,E)$ and write
$\varphi^{*}\nabla^2=\psi^*\nabla^1+\tilde{A}$. One checks that
\[\tilde{A}=\varphi^{-1}A\varphi+\varphi^{-1}\nabla^{1}\varphi-\psi^{-1}\nabla^{1}\psi\]
and using this one also checks that $\varphi^{-1}u\psi$ solves
$(\tilde{D}_{X}+\tilde{A})(\cdot)=0$, where $\tilde{D}$ is the connection
induced by $\psi^{*}\nabla^{1}$. Since $\varphi^{-1}u\psi$ has the same
degree as $u$, $\d$ is well defined.

\begin{Proposition} $\d([\nabla^1],[\nabla^2])$ defines a distance
function on $\T$.
\end{Proposition}

\begin{proof} Suppose $\d([\nabla^1],[\nabla^2])=0$. This means that
there exists $u\in \Omega^{0}(SM,\mbox{\rm Aut}\,\pi^*E)$ which solves
$D^1_{X}u+Au=0$ and $D^1_{V}u=0$. But this last equality means that $u(x,v)$
is independent of $v$. Indeed, consider a unitary trivialization of $E$
over a neighbourhood $V$ of $x$ and write $\nabla^1=d+C$, where
$C$ is a $\mathfrak{u}(n)$-valued 1-form on $V$. Then
\[0=D^1_{V}u=V(u)+[\pi^{*}C(V),u]=V(u).\]
This implies that we may write $u=w\circ\pi$, where $w\in \Omega^{0}(M,\mbox{\rm Aut}\,E)$.
But
\[(D^1_{X}u)(x,v)=(D^1_{X}w\circ\pi)(x,v)=(\nabla^{1}_{d\pi(X)}w)(x)=\nabla^{1}_{v}w.\]
Thus $\nabla^{1}w+Aw=0$, which combined with $\nabla^{2}w=\nabla^{1}w+[A,w]$, implies
that $w^{*}\nabla^2=\nabla^1$. Hence $[\nabla^{1}]=[\nabla^{2}]$ as desired.

To show that $\d$ is symmetric, it suffices to note that if $u$ solves $D_{X}^{1}u+Au=0$, then
$u^{*}$ solves $D^{2}_{X}u^{*}-Au^{*}=0$ and that $u$ and $u^{*}$ have the same degree.

In order to prove the triangle inequality, consider $\nabla^2=\nabla^1+A$ with $D^{1}_{X}u+Au=0$,
and $\nabla^{3}=\nabla^{1}+B$ with $D^{1}_{X}w+Bw=0$. Obviously $\nabla^3=\nabla^2+(B-A)$.
Using that $D^2=D^1+[\pi^*A,\,\cdot]$ and that $D^{1}_{X}u^{*}=u^{*}A$, we compute
\begin{align*}
D_{X}^{2}(wu^{*})&=(D^{2}_{X}w)u^{*}+w(D_{X}^{2}u^{*})\\
&=(D^{1}_{X}w+Aw-wA)u^{*}+w(D^{1}_{X}u^{*}+Au^{*}-u^{*}A)\\
&=(A-B)wu^{*}
\end{align*}
and since $\deg(wu^{*})\leq \deg(u)+\deg(w)$ it follows that
$\d([\nabla^{3}],[\nabla^{2}])\leq \d([\nabla^{3}],[\nabla^{1}])+\d([\nabla^{2}],[\nabla^{1}])$.

\end{proof}

\noindent{\it Proof of Theorem A.} Let us apply Theorem \ref{thm:deg} when $\nabla=\nabla^2$ and
$\nabla^1$ is the trivial connection $d$.  The hypothesis of $\pm i\star F_{\nabla}(x)-K(x)\,\mbox{\rm Id}$
being positive definite for all $x\in M$ implies that $\d([\nabla],[d])=0$.
Thus $\nabla$ is gauge equivalent to the trivial connection.

\section{Proof of Theorem B}

\subsection{Levi-Civita ghosts} As in the introduction, let $K$ be the canonical line bundle and $K^s$ with $s\in \Z$ be its tensor powers (if $s=0$ we get the trivial bundle). The powers $K^s$ for $s\neq 0$ carry the Levi-Civita connection
which we denote by $\nabla_{\ell}^{s}$. If $s=0$ we understand that
this is the trivial connection.
Given an $n$-tuple of integers $S:=(s_{1},\dots,s_{n})$, the connection
\[\nabla_{\ell}^{S}:=\nabla_{\ell}^{s_{1}}\oplus\cdots \oplus\nabla_{\ell}^{s_{n}}\]
defines a transparent unitary connection on the bundle
$E_{S}:=K^{s_{1}}\oplus\cdots\oplus K^{s_{n}}$.
Clearly $c_{1}(E_{S})=(2\texttt{g}-2)(s_{1}+\dots+s_{n})$ and any complex vector bundle $E$
supporting a transparent connection is isomorphic to $E_S$ for $S$ such that
$c_{1}(E)=c_{1}(E_{S})$.

Now let $E$ be a Hermitian vector bundle and consider a unitary isomorphism
$\tau:E\to E_{S}$, where $S$ is such that $c_{1}(E)=c_{1}(E_{S})$.
The unitary connection $\tau^*\nabla_{\ell}^S$ is a transparent
connection on $E$ and its gauge equivalence class, denoted by $[S]$, is independent
of $\tau$. Note that $[S_1]=[S_2]$ if and only if
$S_1$ and $S_2$ coincide up to a permutation.

The next lemma will allows us to see these ghosts in a different form, more
appropriate for our purposes.

 Let $L$ be a $\mathfrak{u}(n)$-valued 1-form
on $SM$. It defines a unitary connection $d_{L}:=d+L$ on the trivial
bundle $SM\times \C^n$.

\begin{Lemma} Suppose $L(X)=L(H)=0$ and $L(V)=c$, where $c\in \mathfrak{u}(n)$
is a constant matrix such that $e^{2\pi\,c}=\mbox{\rm Id}$.
Let $is_{1},\dots,is_{n}$ be the eigenvalues of $-c$, where $s_{k}\in \Z$.
Set $S=(s_{1},\dots,s_{n})$. Then, there exists a unitary
trivialization $\psi:\pi^*E_{S}\to SM\times \C^n$ such that
$\psi^{*}(d_{L})=\pi^{*}\nabla_{\ell}^S$.
\label{lemma:ghosts}
\end{Lemma}

\begin{proof} Since $c\in \mathfrak{u}(n)$, we can find a matrix $a\in U(n)$
such that $-a^{-1}ca$ is a diagonal matrix with entries $is_{1},\dots,is_{n}$.
Hence we might as well assume that $-c$ has already this diagonal form.
It is now clear that it suffices to prove the lemma for the case $n=1$
and we let $s:=s_{1}\in \Z$. The case $s=0$ is obvious and we will prove
the lemma for $s<0$ (the case $s>0$ is similar). If $s<0$, then
$K^s$ is just $(TM)^{\otimes m}$, where $m=-s>0$.
Given $v\in T_{x}M$, we let $v^{m}\in (TM)^{\otimes m} $ be the tensor
product of $v$ with itself $m$ times.

We define a unitary trivialization $\psi:\pi^{*}(TM)^{\otimes m}\to SM\times \C$ as follows. Given $(x,v)\in SM$ and $w\in (T_{x}M)^{\otimes m}$ we set
$\psi(x,v,w)=(x,v,\lambda)$, where $\lambda\in\C$ is the unique
number such that $w=\lambda v^{m}$. Note that the Riemannian metric on
$M$ determines the unitary structure on $(T_{x}M)^{\otimes m}$.
The real 2-dimensional tangent space $T_{x}M$ carries the complex structure $iv$ that rotates a vector $v\in T_{x}M$ by $\pi/2$ according to the orientation of the surface.
We will show that $(\psi^{-1})^*(\pi^{*}\nabla_{\ell}^{s})=d_{L}$.

Let $\xi \in\Omega^{0}(SM,\pi^*(T_{x}M)^{\otimes m})$ be the section given by
$\xi(x,v)=v^m$. Consider a smooth function
$f:SM\to \C$ and note that $(\psi^{-1})^{*}f=f\xi$.
By the definition of the Levi-Civita connection
\[(\pi^{*}\nabla_{\ell}^{s})_{X}(\xi)=(\pi^{*}\nabla_{\ell}^{s})_{H}(\xi)=0\,\]
and thus
\[(\pi^{*}\nabla_{\ell}^{s})_{X}(f\xi)=X(f)\xi=(\psi^{-1})^{*}(d_{L,X}(f)),\]
\[(\pi^{*}\nabla_{\ell}^{s})_{H}(f\xi)=H(f)\xi=(\psi^{-1})^{*}(d_{L,H}(f)).\]
We finally check what happens on $V$. Note that for any affine connection
$\nabla$ on $TM$ we have $\pi^*\nabla_{V}(v)=iv$. Using the definition
of the induced connection on a tensor product we deduce
\[(\pi^{*}\nabla_{\ell}^{s})_{V}(\xi)=mi\xi,\]
hence
\[(\pi^{*}\nabla_{\ell}^s)_{V}(f\xi)=V(f)\xi+fmi\xi.\]
On the other hand
\[d_{L,V}(f)=V(f)+imf\]
and the lemma follows.

\end{proof}

The next lemma, like the previous one, does not need any curvature assumption; 
only that we are working on a surface which is not a torus.
The relation of being $V$-cohomologous is an equivalence relation and given
$f:SM\to \mathfrak{u}(n)$, let $[f]_{V}$ denote the class of $f$.

\begin{Lemma} Let $c_{1},c_{2}\in \mathfrak{u}(n)$ be two matrices
such that $e^{2\pi c_{k}}=\mbox{\rm Id}$ for $k=1,2$.
Then $[c_{1}]_{V}=[c_{2}]_{V}$ if and only if 
$\mbox{\rm tr}(c_1)=\mbox{\rm tr}(c_2)$.
\label{lemma:traces}
\end{Lemma}

\begin{proof} Suppose first that $\mbox{\rm tr}(c_1)=\mbox{\rm tr}(c_2)$.
The matrix $c_k$ determines a bundle $E_{S_{k}}$
and let $\psi_{k}:\pi^*E_{S_{k}}\to SM\times \C^n$ be the unitary
trivialization given by the previous lemma. By hypothesis, we may take
a unitary isomorphism $\phi: E_{S_{1}}\to E_{S_{2}}$ and let
$\rho: \pi^*E_{S_{1}}\to \pi^*E_{S_{2}}$ be the induced isomorphism,
$\rho(x,v,\xi)=(x,v,\phi_{x}(\xi))$.
Let us write $\varphi:=\psi_2\circ\rho\circ\psi_{1}^{-1}(x,v,a)=(x,v,w(x,v)a)$
where $w:SM\to U(n)$ and $a\in\C^n$.
Let $G$ be the unique $\mathfrak{u}(n)$-valued 1-form on $SM$ such that
$\psi_{1}^*(d_{G})=\rho^*\pi^*\nabla^{S_{2}}_{\ell}=\pi^*\phi^*\nabla^{S_{2}}_{\ell}$.

Write $\phi^*\nabla^{S_{2}}_{\ell}=\nabla_{\ell}^{S_{1}}+A$. 
Since $\pi^*\nabla_{\ell}^{S_{1}}=\psi_{1}^*(d_{L_{1}})$ we must have
$G=L_{1}+\psi_{1}\,\pi^*A\,\psi_{1}^{-1}$ which gives $G(V)=L_{1}(V)=c_{1}$.
But $\varphi^*(d_{L_{2}})=d_{G}$, that is,
$G=w^{-1}dw+w^{-1}L_{2}w$. Applying the last equality to $V$ we derive
$c_1=w^{-1}V(w)+w^{-1}c_{2}w$, i.e., $[c_{1}]_{V}=[c_{2}]_{V}$.

Suppose now that there is $w:SM\to U(n)$ such that
$c_1=w^{-1}V(w)+w^{-1}c_{2}w$. Taking traces
\[\mbox{\rm tr}(c_1)-\mbox{\rm tr}(c_2)=h^{-1}V(h),\]
where $h:=\det w:SM\to S^1$. Arguing as in the proof of Theorem \ref{thm:abelian}, the function $h$ gives rise to a real-valued closed 1-form in $SM$ given by
$-ih^{-1}dh$. Since $\pi^*:H^1(M,\re)\to H^{1}(SM,\re)$
is an isomorphism (this follows easily from the Gysin sequence, since
$M$ is not the 2-torus), there exists a closed 1-form $\omega$ in
$M$ and a smooth function $f:SM\to\re$ such that
\[-ih^{-1}dh=\pi^*\omega+df.\]
Applying this equality to $V$ we derive
\[iV(f)=\mbox{\rm tr}(c_1)-\mbox{\rm tr}(c_2)\]
which clearly implies $\mbox{\rm tr}(c_1)=\mbox{\rm tr}(c_2)$.

\end{proof}

The next lemma is not needed in what follows, but it illustrates
the distance in $\mathcal T$.

\begin{Lemma} Suppose $\sum_{k=1}^{n}s_k=0$, so that $\nabla^{S}_{\ell}$ induces a connection
on the trivial bundle. Then $\d([S],[d])=\max|s_k|$, where $d$ is the trivial
connection.
\label{lemma:dist}
\end{Lemma}

\begin{proof} Let $\phi:E_{S}\to M\times \C^n$ be a unitary trivialization and let
$\rho:\pi^{*}E_{S}\to SM\times \C^n$ be the induced unitary trivialization,
$\rho(x,v,\xi)=(x,v,\phi_{x}(\xi))$. Let $A$ be the unique $\mathfrak{u}(n)$-valued
1-form on $M$ given by $\phi^*(d_{A})=\nabla_{\ell}^{S}$, where $d_{A}=d+A$.
Of course, we also have $\rho^{*}(d_{\pi^{*}A})=\pi^{*}\nabla^{S}_{\ell}$.
By Lemma \ref{lemma:ghosts} there is a unitary trivialization $\psi: \pi^{*}E_{S}\to SM\times \C^n$
such that $\psi^{*}(d_{L})=\pi^{*}\nabla_{\ell}^{S}$. 
Hence $(\rho\psi^{-1})^{*}(d_{\pi^{*}A})=d_{L}$.
In other words, if we write $\rho\psi^{-1}(x,v,\xi)=(x,v,u(x,v)\xi)$, where
$u:SM\to U(n)$, then $L=u^{-1}du+u^{-1}\,\pi^{*}A\,u$. Since $L(X)=0$, $u$ solves
$X(u)+Au=0$. An inspection of the construction of $\psi$ in Lemma \ref{lemma:ghosts}
reveals that $\psi$ has polynomial dependence on the velocities with degree
given by $\max|s_{k}|$. It follows that $\deg(u)=\max|s_{k}|$ and thus
$\d([S],[d])\leq \max|s_k|$. Finally note that equality must hold
since if $w$ is another solution of $X(u)+Au=0$, then $u^{*}w$ must be constant.
Indeed, a simple calculation shows that $X(u^{*}w)=0$ and the claim follows from the transitivity
of the geodesic flow of $X$.

\end{proof}

\subsection{Proof of Theorem B} (Forward direction.) The matrix $c$ determines
a bundle $E_S$ and by considering a unitary isomorphism $\tau:E\to E_S$
we may suppose $E=E_{S}$.
Let $\nabla$ be a transparent connection on $E_{S}$ and let
$C$ be its associated cocycle in $\pi^*E_{S}$.
Let $\psi:\pi^*E_{S}\to SM\times \C^n$ be the unitary trivialization
given by Lemma \ref{lemma:ghosts}.

 Write
\[\psi\,C(x,v,t)\,\psi^{-1}(x,v,a)=(\phi_{t}(x,v),D(x,v,t)a),\] where $D:SM\times\re\to U(n)$
is a cocycle as in Definition \ref{def1}. By the Livsic theorem \ref{livsic} there exists
a smooth function $u:SM\to U(n)$ such that $D(x,v,t)=u(\phi_{t}(x,v))u^{-1}(x,v)$.
Let $\Gamma:\re \to SM$ be $\Gamma(t)=\phi_{t}(x,v)$. By the definition
of $C$, $\Gamma^{*}\pi^*\nabla(t\mapsto C(x,v,t)\xi)=0$ for any $\xi\in E_{S}(x)$. Now let
$G$ be the unique $\mathfrak{u}(n)$-valued 1-form on $SM$ such that
$\psi^{*}(d_{G})=\pi^*\nabla$, where $d_{G}=d+G$.
Then $\Gamma^{*}d_G(t\mapsto D(x,v,t)a)=0$ for all $a\in \C^n$. Equivalently
\[\frac{dD}{dt}+G(X)D=0\]
and setting $t=0$, we obtain: $X(u)+G(X)u=0$.

As in the proof of Lemma \ref{lemma:traces}, write $\nabla=\nabla_{\ell}^S+A$. Since $\pi^*\nabla_{\ell}^S=\psi^*(d_{L})$ we must have $G=L+\psi\,\pi^*A\,\psi^{-1}$ which gives $G(V)=L(V)=c$.

Now let us set $B:=u^{-1}du+u^{-1}Gu$. Then $d_{G}$ and $d_{B}$ are gauge equivalent,
but $B(X)=0$.

Since $F_{\pi^*\nabla}(\cdot\,,V)=0$, we must also have
$F_{B}(\cdot\,,V)=0$. Using that $F_{B}=dB+B\wedge B$ and $B(X)=0$ we compute
\[F_{B}(X,V)=dB(X,V)+[B(X),B(V)]=dB(X,V).\]
But
\[dB(X,V)=XB(V)-VB(X)-B([X,V])=XB(V)+B(H),\]
hence
\begin{equation}
B(H)=-XB(V).
\label{eq:cuXV}
\end{equation}
We also compute
\[F_{B}(H,V)=dB(H,V)+[B(H),B(V)],\]
and
\[dB(H,V)=HB(V)-VB(H)-B([H,V])=HB(V)-VB(H),\]
hence
\begin{equation}
HB(V)-VB(H)+[B(H),B(V)]=0.
\label{eq:cuHV}
\end{equation}
Combining (\ref{eq:cuXV}) and (\ref{eq:cuHV}) we derive the following
non-linear PDE for $f:=B(V)$
\begin{equation}
H(f)+VX(f)-[X(f),f]=0.
\label{eq:nlpde}
\end{equation}
This is precisely equation (\ref{eq:keypde}) in the Introduction and since
$f=B(V)=u^{-1}V(u)+u^{-1}cu$ it follows that $f\in \mathcal H_{c}$.
Note that $f$ is defined exclusively in terms of $u$ and $c$ and
$u$ must solve $X(u)+G(X)u=0$. However, up to right multiplication
by an element $q\in U(n)$, there is only one such solution.
Indeed, if $w$ is another solution, then $X(u^*w)=0$ and by transitivity
of the geodesic flow there is $q\in U(n)$ such that $w=uq$.
This implies that $f$ is uniquely defined in $\mathcal H_{c}/U(n)$.
To complete the correspondence in the forward direction, we must
check that if we consider a connection gauge equivalent to $\nabla$
we obtain the same $f$. A connection $\nabla^1$ gauge equivalent
to $\nabla$ determines a connection $d_{G_1}$ in $SM$ gauge equivalent
to $d_{G}$. In other words, there is a smooth function $r:SM\to U(n)$
such that $G_1=r^{*}dr+r^{*}Gr$. But if $u$ solves
$X(u)+G(X)u=0$, then $r^{*}u$ solves $X(w)+G_1(X)w=0$ (unique up to 
multiplication on the right by an element in $U(n)$).
Next observe that
$$G_1(V)=c=r^*V(r)+r^*cr$$ and
$$f_{1}=u^*rV(r^*u)+u^*rcr^*u=u^*V(u)+u^*(rV(r^*)+rcr^*)u=f$$
thus obtaining a well defined map $\T\mapsto \mathcal H_{c}/U(n)$.

\medskip

(Backward direction.) Suppose now that we have a solution $f$ of
(\ref{eq:nlpde}) such that there is $u:SM\to U(n)$ with
$f=u^*V(u)+u^*cu$. Define a $\mathfrak u(n)$-valued 1-form $G$
on $SM$ by setting:
\begin{align*}
G(X)&=-X(u)u^*,\\
G(H)&=-uX(f)u^*-H(u)u^*,\\
G(V)&=c,
\end{align*}
and define an element $\A\in \Omega^{1}(SM,\mbox{\rm ad}\,\pi^*E_{S})$
by 
\[\A:=\psi^{-1}(G-L)\psi.\]
Clearly $\A(V)=0$ and we wish to show that there exists
$A\in \Omega^{1}(M,\mbox{\rm ad}\,E_{S})$ such that $\A=\pi^*A$.
For this, it suffices to show that $D_{V}\A(X)=\A(H)$ and
$D_{V}\A(H)=-\A(X)$, where $D$ here stands for the connection
induced by $\nabla_{\ell}^S$. Equivalently, using the unitary
isomorphism $\psi$, we are required to show that
$D_{V}^LG(X)=G(H)$ and $D_{V}G(H)=-G(X)$, where $D^L$ is induced
by $d_{L}$. Explicitly this means $V(G(X))+[c,G(X)]=G(H)$ and
$V(G(H))+[c,G(H)]=-G(X)$.
Using the definition of $G(X)$, the structure equations of the metric
and $uf=V(u)+cu$ we compute:
\begin{align*}
V(G(X))&=-VX(u)u^*-X(u)V(u^*)\\
&=-XV(u)u^*-H(u)u^*-X(u)V(u^*)\\
&=-X(uf-cu)u^*-H(u)u^*-X(u)(u^*c-fu^*)\\
&=-uX(f)u^*-H(u)u^*+[c,X(u)u^*]\\
&=G(H)-[c,G(X)].
\end{align*}
Similarly we compute (we omit some of the details)
\begin{align*}
V(G(H))&=-V(u)X(f)u^*-uVX(f)u^*-uX(f)V(u^*)-VH(u)u^*-H(u)V(u^*)\\
&=X(u)u^*+u([X(f),f]-VX(f)-H(f))u^*+[c,uX(f)u^*]+[c,H(u)u^*]\\
&=-G(X)-[c,G(H)]+u([X(f),f]-VX(f)-H(f))u^*.\\
\end{align*}
Thus, if $f$ satisfies equation (\ref{eq:nlpde}) we have
$V(G(H))+[c,G(H)]=-G(X)$ as desired.
Since $\A=\pi^*A$, $\nabla:=\nabla_{\ell}^S+A$ defines a transparent
connection on $E_S$. To complete the backward correspondence we need
to discuss what happens when we have another solution $w$ to the
equation $f=u^*V(u)+u^*cu$. In this case $w$ determines a connection
$d_{G_{1}}$ on $SM$ and it is straightforward to check that if
we let $r:=uw^*$, then $G_1=r^*dr+r^*Gr$. In other words, $d_{G}$
and $d_{G_{1}}$ are gauge equivalent in $SM$. It follows
that $\pi^*\nabla$ and $\pi^*\nabla^1$ are gauge equivalent. But this implies
that $\nabla$ and $\nabla^1$ are gauge equivalent.
Indeed, suppose $\pi^*\nabla$ and $\pi^*\nabla^1$
are gauge equivalent via $\varphi\in \Omega^0(SM,\mbox{\rm Aut}\,\pi^*E_S)$,
 i.e. 
$\pi^{*}\nabla^1=\varphi^{-1}D\varphi+\pi^*\nabla$.
Apply this to $V$ to obtain $D_{V}\varphi=0$, so $\varphi$ only depends
on the base point $x$. Applying the equality to $X$ we deduce that
there is $\varphi\in\Omega^0(M,\mbox{\rm Aut}\,E_S)$ such that
$\nabla^1=\varphi^{-1}\nabla\varphi+\nabla$
and thus $\nabla^1$ and $\nabla$ are gauge equivalent. 
The proof of Theorem B is now complete.

\medskip

\noindent{\bf Addendum to Theorem B.}
We claim that $\mbox{\rm tr}(f)=\mbox{\rm tr}(c)$ and thus
 $\mbox{\rm tr}(f)$ is constant and determined by the topology
of $E$.
Consider the transparent connections induced by $\nabla$ and $\nabla_{\ell}^S$
on the line bundle $\det E_{S}$. By Theorem \ref{thm:abelian}
they must be gauge equivalent; in other words there is a smooth
function $g:M\to S^1$ such that $\mbox{\rm tr}(A)=dg/g$.
Recall that $X(u)+G(X)u=0$ and $G(X)=\psi A_{x}(v)\psi^{-1}$.
Hence $\mbox{\rm tr}(G(X))=\mbox{\rm tr}(A)=X(g)/g$.
Since $X(\det u)=\det u\,\mbox{\rm tr}(u^{*}X(u))$ we derive
$X(\det u)=\det u(-X(g)/g)$ and thus $X(g\det u)=0$.
By transitivity of the geodesic flow $g\det u$ is a constant
and hence $V(\det u)=0$. But this is equivalent to $\mbox{\rm tr}(u^*V(u))=0$.
Since $f=u^*V(u)+u^* c u$,  $\mbox{\rm tr}(f)=\mbox{\rm tr}(c)$
as desired.

\begin{Remark}{\rm One can also compute the curvature $F_{B}(X,H)$
of the connection $d_B$ from the theorem.
Using that $B(X)=0$ we derive:
\begin{equation}
F_{B}(X,H)=dB(X,H)= XB(H)-B([X,H])=-X^{2}(f)-Kf.
\label{eq:ode}
\end{equation}
Note that $F_{B}(X,H)$ is conjugate to $\star F_{\nabla}\circ\pi$ via a unitary trivialization.
}
\label{rem:curva}
\end{Remark}

\subsection{Transparent connections at distance one from the trivial connection}

Let $\nabla$ be a transparent connection on the trivial bundle with
$\d([\nabla],[d])=1$. If we follow the proof of Theorem B, we see that
in this case $\psi$ is the identity and if we write
$\nabla=d+A$, then $G=\pi^*A$.
Since $\d([\nabla],[d])=1$, there exists a smooth 
function $u:SM\to U(n)$ such that $u=u_{-1}+u_{0}+u_{1}$ and
$X(u)+Au=0$. Also, $B=u^{-1}du+u^{-1}\,\pi^*A\, u$ and
$f=B(V)=u^{*}V(u)=-V(u^{*})u$.

\begin{Lemma} $f\in \Omega_0$.
\end{Lemma}

\begin{proof} By separating $X(u)+Au=0$ into even and odd parts we
deduce
\[X(u_0)+Au_{0}=0,\]
\[X(u_{-1}+u_{1})+A(u_{-1}+u_{1})=0.\]
These two equations yield
\[X(u^{*}_{0}(u_{-1}+u_{1}))=u_{0}^{*}A(u_{-1}+u_{1})+u_{0}^{*}(-A)(u_{-1}+u_{1})=
0,\]
and since the geodesic flow is transitive $u^{*}_{0}(u_{-1}+u_{1})$
must be constant and thus
\begin{equation}
u_{0}^{*}u_{1}=u_{0}^{*}u_{-1}=0.
\label{eq:u0u1}
\end{equation}
Using the special form of $u$ we derive
\begin{align*}
f&=(u_{-1}^{*}+u_{0}^{*}+u_{1}^{*})(-iu_{-1}+iu_{1})\\
&=i(u^{*}_{-1}u_{1}-u_{-1}^{*}u_{-1}+u_{0}^{*}u_{1}-u_{0}^{*}u_{-1}
+u_{1}^{*}u_{1}-u_{1}^{*}u_{-1}).\\
\end{align*}
Using that $u^{*}u=\mbox{\rm Id}$ we see that the terms
of degree $\pm 2$, $u_{-1}^{*}u_{1}\in \Omega_{2}$ and $u_{1}^{*}u_{-1}\in\Omega_{-2}$ must vanish. Using (\ref{eq:u0u1}) we obtain
\begin{equation*}
f=i(u_{1}^{*}u_{1}-u_{-1}^{*}u_{-1})\in\Omega_{0}.
\end{equation*}

\end{proof}

\begin{Corollary} Suppose the Hermitian
 matrix $\pm i\star F_{\nabla}(x)-2K(x)\,\mbox{\rm Id}$
is positive definite for all $x\in M$. Then $f\in \Omega_0$.
\end{Corollary}

\begin{proof} This follows right away from the last lemma
and Theorem \ref{thm:deg} which implies that $\d([d],[\nabla])\leq 1$.
\end{proof}

Since $f\in\Omega_0$ we can think of $f$ as a function which depends only
 on the base point $x$. Thus $X(f)(x,v)=df_{x}(v)$ and $H(f)(x,v)=df_{x}(iv)$.
Since $VX(f)=XV(f)+H(f)=H(f)$, equation (\ref{eq:nlpde}) gives $2H(f)=[X(f),f]$
and we can rewrite this in terms of matrix valued 1-forms
as 
\begin{equation}
2\star df=[df,f].
\label{eq:pdeonm}
\end{equation}

We discuss this equation in the next subsection. Note that if we wish
$f$ to be $V$-cohomologous to a matrix $c$ as in Theorem B we must
have $e^{2\pi f(x)}=\mbox{\rm Id}$. 

\subsection{Solutions to $2\star df=[df,f]$.}
\label{sub:sol}
If we let $A:=\frac{1}{2}\star df$, then $2\star df=[df,f]$ may be rewritten
as $d_{A}f=0$, so $f$ is covariant constant relative to
the connection $d_{A}$. This implies that $f$ only hits one adjoint orbit
of the adjoint action of $U(n)$ on $\mathfrak{u}(n)$. To see that this is
 the case observe first that
$d\mbox{\rm tr}(f^m)=\mbox{\rm tr}(d_{A}f^{m})=0$ for any $m$ and thus
the eigenvalues of $f(x)$ must be constant (and belong to $i\Z$ if
$e^{2\pi f(x)}=\mbox{\rm Id}$). Also,
the multiplicities of the eigenvalues do not change with $x$. Indeed,
let $\xi\in \C^n$ be an eigenvector of $f(x)$ with eigenvalue $\lambda$
and let $\gamma:[0,1]\to M$ be a curve connecting $x$ to $y$. 
Let $\xi(t)$ be the parallel transport of $\xi$ along $\gamma$. Since
$f(\gamma(t))\xi(t)$ is also parallel ($d_{A}f=0$), it must equal 
$\lambda\xi(t)$ and thus $f(y)\xi(1)=\lambda \xi(1)$ which shows that
parallel transport preserves the eigenspaces of $\lambda$.

Suppose now $f:M\to \mathfrak{su}(2)$. The discussion above implies that
$f^{2}=-\lambda^2\mbox{\rm Id}$ for some constant $\lambda$. This implies that
$df\,f=-f\,df$, so we rewrite $2\star df=[df,f]$ as $\star df=df\,f$.
Applying $\star$ we derive $df=\lambda^2df$. 
Hence if $\lambda^2\neq 1$, $f$ must be 
constant. Let us suppose then that $f^{2}=-\mbox{\rm Id}$, so $f$ hits
the adjoint orbit of $$\left(\begin{array}{cc}
i&0\\
0&-i\\
\end{array}\right)$$ which we denote by $\bf S$ and we identify
with the 2-sphere. For $g\in \bf S$ and $X\in T_{g}\bf S$,
let $J_{g}(X):=Xg$. Clearly $J_{g}^2=-\mbox{\rm Id}$, so $J_g$ is a complex
structure in $\bf S$ and the equation $\star df=df\,f$ simply
says $df_{x}(iv)=J_{f(x)}(df_{x}(v))$, i.e. $f:M\to\bf S$ is a holomorphic map.

We now wish to show that given such a map $f:M\to\bf S$, then
$f$ is $V$-cohomologous to the zero matrix, that is, there exists
$u:SM\to SU(2)$ such that $f=u^*V(u)$. This would show that
$\mathcal H_{0}\cap\Omega_0$ can be identified with the set
of holomorphic maps $f:M\to \mathbb \C\mathbb P^1$
as claimed in the introduction.
 
Consider a map $f:M\to \bf S$ and let $L(x)$ (resp. $U(x)$) be the eigenspace corresponding to the eigenvalue $i$ (resp. $-i$) of $f(x)$. 
We have an orthogonal decomposition $\C^{2}=L(x)\oplus U(x)$ for every
$x\in M$.
Consider sections $\alpha\in \Omega^{1,0}(M,\C)$ and
$\beta\in \Omega^{1,0}(M,\mbox{\rm Hom}(L,U))=\Omega^{1,0}(M,L^*U)$ such that
$|\alpha|^2+|\beta|^2=1$. Such pair of sections always exists; for example, we can choose a section $\tilde{\beta}$ with a finite number of isolated zeros
and then choose $\tilde{\alpha}$ such that it does not vanish on the zeros
of $\tilde{\beta}$. Then we
set $\alpha:=\tilde{\alpha}/(|\tilde{\alpha}|^2+|\tilde{\beta}|^2)^{1/2}$ and $\beta:=\tilde{\beta}/(|\tilde{\alpha}|^2+|\tilde{\beta}|^2)^{1/2}$.

Note that $\bar{\alpha}\in  \Omega^{0,1}(M,\C)$
and $\beta^*\in \Omega^{0,1}(M,\mbox{\rm Hom}(U,L))=\Omega^{0,1}(M,U^*L)$. 
Using the orthogonal decomposition we define $u:SM\to SU(2)$ by
 $$u(x,v)=\left(\begin{array}{cc}
\alpha(x,v)&\beta^*(x,v)\\
-\beta(x,v)&\bar{\alpha}(x,v)\\
\end{array}\right).$$
Clearly $u=u_{-1}+u_{1}$, where
 $$u_{1}=\left(\begin{array}{cc}
\alpha&0\\
-\beta&0\\
\end{array}\right)$$
and
 $$u_{-1}=\left(\begin{array}{cc}
0&\beta^*\\
0&\bar{\alpha}\\
\end{array}\right).$$
It is straightforward to check that $uf=V(u)$.

Combining the discussion above with Theorem B (and its addendum)
we derive:

\begin{Corollary} The set of transparent $U(2)$-connections modulo gauge
transformations at distance
one from the trivial connection is in 1-1 correspondence with
holomorphic maps $f:M\to \C\mathbb P^1$ up to composition with an
orientation preserving isometry of $\C\mathbb P^1$.
\label{cor:last}
\end{Corollary}

\begin{Remark}{\rm We can actually compute the distance $\d([A],[B])$
where $A$ and $B$ define transparent connections at distance
one from the trivial connection. Let $u=u_{-1}+u_{1}$ solve
$X(u)+Au=0$ and let $w=w_{-1}+w_{1}$ solve $X(w)+Bw=0$.
Then $r=wu^*$ solves $X(r)+Br-rA=0$ or equivalently
$D_{X}^{A}r+(B-A)r=0$. In fact it is easy to check using arguments already
used before that $wqu^*$, where $q\in SU(2)$ is a constant matrix,
are all the solutions of $X(r)+Br-rA=0$. Now observe that
$wqu^*=w_{-1}qu^{*}_{-1}+w_{1}qu_{1}^*+w_{-1}qu_{1}^*+w_{1}qu^{*}_{-1}$.
Thus $wqu^*$ has terms only of degree zero or $\pm 2$. It follows
that $\d([A],[B])=2$ unless $[A]=[B]$. Hence the distance induced
via Corollary \ref{cor:last} on the space of holomorphic maps
$f:M\to\C\mathbb P^1$ (modulo $SU(2)$) is just the discrete distance.

}
\end{Remark}

\subsection{The energy estimates method}
\label{sub:last}
In order to deal with equation (\ref{eq:nlpde})
one may try to use the energy estimates method (the Pestov identity)
in the case of matrix valued functions as done by L.B. Vertgeim \cite{V}, Sharafutdinov \cite{Sha} and Finch and Uhlmann \cite{FU}.
However in order to control the non-linear term given by the bracket in (\ref{eq:nlpde}) one ends up requiring some assumption
of smallness on the connection or its curvature.

In our case the relevant integral identity takes virtually the same form as in the case of complex
valued functions; 
we omit its proof here which is a straightforward generalization of the case $n=1$, which may be found
in the form below in \cite[Lemma 2.1]{SU}. Let $f:SM\to \mathbb M_{n}(\C)$ be a smooth
function, where $\mathbb M_{n}(\C)$ denotes the set of $n\times n$ complex matrices. Then

\[2\int_{SM}\langle H(f),VX(f)\rangle\,d\mu=\int_{SM}|H(f)|^2\,d\mu+\int_{SM}|X(f)|^2\,d\mu-\int_{SM}K|V(f)|^2\,d\mu,\]
where $\langle A,B\rangle=\Re\,\mbox{\rm tr}(AB^*)$ for $A,B\in \mathbb M_{n}$.
If $f$ satisfies equation (\ref{eq:nlpde}), then
\begin{equation}
2\int_{SM}\langle H(f),[X(f),f]\rangle\,d\mu=3\int_{SM}|H(f)|^2\,d\mu+\int_{SM}|X(f)|^2\,d\mu-\int_{SM}K|V(f)|^2\,d\mu.
\label{eq:ee}
\end{equation}
The last equality gives right away that $f$ is constant if $n=1$. Indeed, if $K<0$ the right hand side
of the equality is $\geq 0$ and the left hand side vanishes since the bracket must vanish. This
implies $H(f)=X(f)=V(f)=0$ and thus $f$ is constant. For $n\geq 2$ it is not clear how to deal
with the term in the left hand side for arbitrary $f$. Here is an attempt in the spirit of \cite{FU}.

Using the Cauchy-Schwartz inequality we can estimate the left hand side of (\ref{eq:ee}) by
\[2\int_{SM}\langle H(f),[X(f),f]\rangle\,d\mu\leq 2 \max||f||\int_{SM}(|X(f)|^2+|H(f)|^2)\,d\mu,\]
where $||f||$ is the operator norm of $f$. Hence if $2\max ||f||\leq 1$, (\ref{eq:ee}) gives when
$K<0$ that $H(f)=V(f)=0$ and again $f$ must be constant ($X=-[V,H]$). One can now try to estimate $\max||f||$
using the curvature of $\nabla$ and Remark \ref{rem:curva}.
If for example $K=-1$ we can solve the ODE (\ref{eq:ode}) explicitly as
\[2f(x,v)=\int_{0}^{\infty}e^{-s}F_{B}(X,H)(\phi_{s}(x,v))\,ds+\int_{-\infty}^{0}e^{s}F_{B}(X,H)(\phi_{s}(x,v))\,ds.\]
Hence if the operator norm of $\star F_{\nabla}$ is everywhere $\leq 1/2$, so is
$F_{B}(X,H)$, and then $\max ||f||\leq 1/2$.
By the argument above $f$ must be constant. However, this seems to give a weaker result
than Theorem A.

Finally we note that (\ref{eq:ee}) shows that if $f$ is a solution
of (\ref{eq:nlpde}) which is also {\it odd} (i.e. $f(x,-v)=-f(x,v)$), 
then it must be identically zero. Indeed, in this case $H(f)$, $X(f)$
and $VX(f)$ are even functions, but $[X(f),f]$ is odd. It follows
that $[X(f),f]=0$ and by (\ref{eq:ee}), $f$ must be a constant, and thus
identically zero.

\end{document}